\documentclass[12pt]{article}
\usepackage{amsfonts}
\usepackage{amsmath}
\usepackage{graphicx}

\newtheorem{theorem}{Theorem}

\newtheorem{definition}[theorem]{Definition}

\newtheorem{lemma}[theorem]{Lemma}

\newtheorem{remark}[theorem]{Remark}

\newenvironment{proof}[1][Proof]{\noindent\textbf{#1.} }{\ \rule{0.5em}{0.5em}}
\numberwithin{equation}{section}

\newcommand\mycom[2]{\genfrac{}{}{0pt}{}{#1}{#2}}
\input{tcilatex}

\begin{document}

\title{Fractional diffusion-type equations with exponential and logarithmic
differential operators }
\author{Luisa Beghin\thanks{%
Address: Department of Statistical Sciences, Sapienza University of Rome,
P.le A. Moro 5, I-00185 Roma, Italy. e-mail: \texttt{luisa.beghin@uniroma1.it%
}}}
\date{}
\maketitle

\begin{abstract}
\noindent We deal with some extensions of the space-fractional diffusion
equation, which is satisfied by the density of a stable process (see \cite%
{MAI}): the first equation considered here is obtained by adding an
exponential differential operator expressed in terms of the Riesz-Feller
derivative. We prove that this produces a random additional term in the
time-argument of the corresponding stable process, which is represented by
the so-called Poisson process with drift. Analogously, if we add, to the
space-fractional diffusion equation, a logarithmic differential operator
involving the Riesz-derivative, we obtain, as a solution, the transition
semigroup of a stable process subordinated by an independent gamma
subordinator with drift. Finally, we show that a non-linear extension of the
space-fractional diffusion equation is satisfied by the transition density
of the process obtained by time-changing the stable process with an
independent linear birth process with drift.


\end{abstract}

\section{Introduction}

The diffusion equation has been generalized in the fractional sense by many
authors (e.g. \cite{WYS}, \cite{SCH}, \cite{CHE}, \cite{ANG}): in particular
\cite{MEE} considers the time-fractional Cauchy problems, while in \cite{LUC}
the order of both time and space derivatives is fractional. Later, in \cite%
{MAI} and \cite{MAI2}, the time and space fractional diffusion equation was
studied and solved analytically, also in the asymmetric case. The
probabilistic expression of the solution to the diffusion equation with
time-derivative of fractional order $\nu $ is given in \cite{ORB}, in terms
of iterated stable processes (in particular, for $\nu =1/2^{n},$ $n\in
\mathbb{N}$, the $n$-times iterated Brownian motion).

The term 'anomalous diffusion' usually indicates a diffusive process that
does not follow the behavior of classical Gaussian diffusions. In the real
world, anomalous diffusions are observed, for example, in turbulent plasma
transport, photon diffusion, cell migration and so on. When a fractional
derivative replaces the second-order derivative in a diffusion model, we get
the so called superdiffusion. On the other hand, considering fractional time
derivatives produces anomalous subdiffusion, slower than a classical
diffusion. The stochastic time-space fractional heat-type equation has been
treated in \cite{MIJ}. For applications to physical and financial problems,
see also \cite{KOC}, \cite{MEZ}, \cite{SCA}.

We consider here extensions of the following space-fractional diffusion
equation, i.e.
\begin{equation*}
\partial _{t}u(x,t)=\mathcal{D}_{x}^{\alpha ,\theta }u(x,t),
\end{equation*}%
where $\mathcal{D}_{x}^{\alpha ,\theta }$ is the Riesz-Feller derivative of
order $\alpha \in (0,2]$, defined below. In particular we introduce in the
above equation additional terms represented by the so-called fractional
exponential (or shift) operator $\mathcal{O}_{c,x}^{\alpha ,\theta }$ or the
fractional logarithmic operator $\mathcal{P}_{c,x}^{\alpha }$ (see Def.1 and
Def.4 below)$.$ We are thus led to study the following equations, again for $%
\alpha \in (0,2],$%
\begin{equation}
\partial _{t}u(x,t)=\left[ a\mathcal{D}_{x}^{\alpha ,\theta }+\lambda (I-%
\mathcal{O}_{-1,x}^{\alpha ,\theta })\right] u(x,t)  \label{pp}
\end{equation}%
\begin{equation}
\partial _{t}u(x,t)=\left[ a\mathcal{D}_{x}^{\alpha }+\mu \mathcal{P}%
_{1/\rho ,x}^{\alpha }\right] u(x,t)  \label{aa}
\end{equation}%
under appropriate initial and boundary conditions. We prove that the
solution to equation (\ref{pp}) coincides with the transition semigroup of
the subordinated process defined as $\mathcal{S}_{\alpha ,\theta }(at+N(t)),$
$t\geq 0$, where $\mathcal{S}_{\alpha ,\theta }$ is an $\alpha $-stable
process and $N$ is an independent Poisson subordinator, with parameter $%
\lambda .$ In the second case, we prove instead that equation (\ref{aa}) is
satisfied by the transition semigroup of another subordinated $\alpha $%
-stable process defined as $\mathcal{S}_{\alpha }(at+\Gamma (t)),$ where $%
\Gamma (t),$ $t\geq 0$, is an independent gamma subordinator, with scale
parameter $\mu >0$. However, in both cases, the processes obtained are, for
any $\alpha \in (0,2)$, pure jump models, while, only for $\alpha =2,$ they
have a jump-diffusion behavior. In particular, we notice that $\mathcal{S}%
_{\alpha }(at+\Gamma (t))$ reduces, for $\alpha =2$ and $a=0$, to the
well-known Variance Gamma (VG) process. Jump-diffusions and VG processes are
applied in finance, in particular for asset pricing (see e.g. \cite{CON}).
For a general $\alpha \in (0,2)$, the process $\mathcal{S}_{\alpha
}(at+\Gamma (t))$ can be considered as a generalization of both stable and
geometric stable processes (see, for example, \cite{KOZ2}), to which it
reduces in special cases.

In the last section we prove that a non-linear analogue of (\ref{pp}) is
satisfied by the transition density of a stable process time-changed by an
independent linear birth process with drift.

Therefore, in all these cases, the additional operator introduced in the
fractional diffusion equation entails the appearance of a random element in
the time argument of the corresponding process. This additional random
element is represented, in the case of the fractional exponential operator,
by the Poisson or birth processes, while, for the fractional logarithmic
operator, it is given by the gamma process.

We remark that both equations (\ref{pp}) and (\ref{aa}) are connected to L%
\'{e}vy processes, while, only in the nonlinear case, we obtain a process
whose finite distributions are not infinitely divisible, even though it
still enjoys the Markov property.

\quad

We now introduce the notation and the basic definitions that we will use
throughout the paper.

Let $X:=X(t),t\geq 0$ be a one-dimensional L\'{e}vy process in $\mathbb{R}$
and $f$ be in the Schwartz space $S\mathcal{(\mathbb{R})}$ of rapidly
decreasing functions. Then we denote by $\widetilde{f}(\xi )$ the Fourier
transform of $f$, i.e. $\widetilde{f}(\xi ):=\mathcal{F}\left\{ f(x);\xi
\right\} =\int_{-\infty }^{+\infty }e^{ix\xi }f(x)dx,$ and by $T_{t}$ the
Feller semigroup associated to $X$, i.e.%
\begin{equation*}
(T_{t}f)(x)=\mathbb{E}f(x-X(t)),
\end{equation*}%
for $f\in C_{0}(\mathbb{R})$, the real Banach space of continuous functions
satisfying $\lim_{x\rightarrow \pm \infty }f(x)=0.$ The symbol of $T_{t}$ is
given by $\widehat{T}_{t}=e^{-t\eta },$ i.e.%
\begin{equation*}
\mathcal{F}\left\{ T_{t}f;\xi \right\} =e^{-t\eta (\xi )}\widetilde{f}(\xi
),\quad \xi \in \mathbb{R},
\end{equation*}%
while the L\'{e}vy (or characteristic) exponent of $X$ will be denoted by%
\begin{equation*}
\eta _{X}(\xi )=\frac{1}{t}\ln \left( \mathbb{E}e^{i\xi X(t)}\right) ,
\end{equation*}%
(see, for example, \cite{APPL}). Let $\mathcal{A}$ be the
pseudo-differential operator such that \ $T_{t}=e^{t\mathcal{A}}$, then we
denote by $\widehat{\mathcal{A}}$ its symbol, i.e.%
\begin{equation*}
\mathcal{F}\left\{ \mathcal{A}f(x);\xi \right\} =\widehat{\mathcal{A}}(\xi )%
\widetilde{f}(\xi )=-\eta (\xi )\widetilde{f}(\xi ),\quad \xi \in \mathbb{R}.
\end{equation*}%
We will consider the $\alpha $-stable process $\mathcal{S}_{\alpha ,\theta
}(t),$ $t\geq 0,$ with transition density $p_{\alpha ,\theta }(x;t)$ and
characteristic function
\begin{equation}
\Phi _{\mathcal{S}_{\alpha ,\theta }(t)}(\xi ):=\mathbb{E}e^{i\xi \mathcal{S}%
_{\alpha ,\theta }(t)}=\exp \{-t|\xi |^{\alpha }\sigma ^{\alpha }\omega
_{\alpha ,\theta }(\xi )\},\qquad \xi \in \mathbb{R},\;\alpha \in
(0,2],\;\sigma >0,  \label{ch2}
\end{equation}%
where $\theta =\frac{2}{\pi }\arctan \left[ -\beta \tan \pi \alpha /2\right]
$ and%
\begin{equation*}
\omega _{\alpha ,\theta }(\xi ):=\left\{
\begin{array}{c}
1-i\beta sign(\xi )\tan (\pi \alpha /2),\quad \text{if }\alpha \neq 1 \\
1+2i\beta sign(\xi )\log |\xi |/\pi ,\quad \text{if }\alpha =1%
\end{array}%
\right. ,\;\beta \in \lbrack -1,1],
\end{equation*}%
(see \cite{SAMO}, for general references on stable r.v.'s). By assuming $%
\sigma =\left( \cos \pi \theta /2\right) ^{1/\alpha }$, we can write (\ref%
{ch2}) as $\Phi _{\mathcal{S}_{\alpha ,\theta }(t)}(\xi )=\exp \{-t|\xi
|^{\alpha }e^{isign(\xi )\pi \theta /2}\}.$

We will use the \emph{Riesz-Feller (RF) fractional derivative}, which is
defined by means of its Fourier transform. Thus we consider the space $L^{c}(%
\mathcal{I})$ of functions for which the Riemann improper integral on any
open interval $\mathcal{I}$ absolutely converges (see \cite{MAR}). For any $%
f\in L^{c}(\mathbb{R}),$ the RF fractional derivative is defined as

\begin{equation}
\mathcal{F}\left\{ \mathcal{D}_{x}^{\alpha ,\theta }f(x);\xi \right\} =-\psi
_{\alpha ,\theta }(\xi )\mathcal{F}\left\{ f(x);\xi \right\} ,\quad \alpha
\in (0,2],\quad |\theta |\leq \min \{\alpha ,2-\alpha \},  \label{uno}
\end{equation}%
with symbol
\begin{equation}
\widehat{\mathcal{D}_{x}^{\alpha ,\theta }}(\xi )=-\psi _{\alpha ,\theta
}(\xi ):=-|\xi |^{\alpha }e^{i\,sign(\xi )\theta \pi /2},  \label{due}
\end{equation}%
(see \cite{MAI} and \cite{KIL}, p.359, up to the sign). Thus the expression
in (\ref{due}) coincides with the L\'{e}vy exponent of the $\alpha $-stable
r.v. $\mathcal{S}_{\alpha ,\theta }:=\mathcal{S}_{\alpha ,\theta }(1).$ As a
consequence, the RF derivative coincides with the generator of the stable
process $\mathcal{S}_{\alpha ,\theta },$ since it is proved in \cite{MAI}
that its density $p_{\alpha ,\theta }(x;t)$ solves the following problem%
\begin{equation}
\partial _{t}u(x,t)=\mathcal{D}_{x}^{\alpha ,\theta }u(x,t),\qquad
u(x,0)=f(x),\quad u(\pm \infty ,t)=0,\qquad x\in \mathbb{R},\text{ }t\geq 0,
\label{mai}
\end{equation}%
with $f\in L^{c}(\mathbb{R}).$

For $\theta =0$, which corresponds to the symmetric case (i.e. for $\beta
=0) $, the RF derivative reduces to the \textit{Riesz derivative}, which, in
its regularized form (valid also for $\alpha =1)$, can be written as
\begin{equation}
\mathcal{D}_{x}^{\alpha ,0}u(x)=\left\{
\begin{array}{l}
\frac{\Gamma (1+\alpha )}{\pi }\sin \left( \frac{\alpha \pi }{2}\right)
\int_{0}^{+\infty }\frac{f(x+z)-2f(x)+f(x-z)}{z^{1+\alpha }}dz,\quad \alpha
\in (0,2) \\
\partial _{x}^{2},\quad \alpha =2%
\end{array}%
\right. ,  \label{rie}
\end{equation}%
(see \cite{GOR}, p.341, for details). We will denote the latter simply as $%
\mathcal{D}_{x}^{\alpha }.$

For $\alpha \in (0,1)$ and $\theta =-\alpha $, which corresponds to the case
of the stable subordinator (with asymmetry parameter $\beta =1$), the
fractional derivative $\mathcal{D}_{x}^{\alpha ,\theta }$ coincides with the
(left-sided) Riemann-Liouville derivative, i.e.%
\begin{equation}
\mathcal{D}_{x}^{\alpha ,-\alpha }u(x)=\left\{
\begin{array}{l}
\frac{(-1)}{\Gamma (1-\alpha )}\frac{d}{dx}\int_{x}^{+\infty }\frac{f(z)}{%
(x-z)^{\alpha }}dz,\quad \alpha \in (0,1) \\
-\partial _{x},,\quad \alpha =1%
\end{array}%
\right. .  \label{crr}
\end{equation}%
Indeed
\begin{equation}
\widehat{\mathcal{D}_{x}^{\alpha ,-\alpha }}(\xi )=-\psi _{\alpha ,-\alpha
}(\xi )=-|\xi |^{\alpha }e^{-isign(\xi )\alpha \pi /2}=-(-i\xi )^{\alpha },
\label{sy}
\end{equation}%
see \cite{GOR}, p.333. In this case it is well-known that equation (\ref{mai}%
) is satisfied by the density of the $\alpha $-stable subordinator.

\section{Preliminary results}

We now introduce the following pseudo-differential operators, defined in
terms of the RF and the Riesz fractional derivatives (\ref{uno}) and (\ref%
{rie}), respectively.

\begin{definition}
\textbf{(Fractional shift or exponential operator)} Let $f\in L^{c}(\mathbb{R%
})$ be a function s.t. $g_{j}f:=\underbrace{\mathcal{D}_{x}^{\alpha ,\theta
}...\mathcal{D}_{x}^{\alpha ,\theta }}_{j-times}f\in L^{c}(\mathbb{R})$, for
any $j=0,1,...$. , then%
\begin{equation}
\mathcal{O}_{c,x}^{\alpha ,\theta }f(x):=\sum_{n=0}^{\infty }\frac{c^{n}}{n!}%
\underbrace{\mathcal{D}_{x}^{\alpha ,\theta }...\mathcal{D}_{x}^{\alpha
,\theta }}_{n-times}f(x),\qquad c\in \mathbb{R},,\quad \alpha \in
(0,2],\quad |\theta |\leq \min \{\alpha ,2-\alpha \}.  \label{fr1}
\end{equation}%
provided that the series converges uniformly.

\begin{lemma}
The symbol of (\ref{fr1}) is given, for $c\in \mathbb{R}$, by%
\begin{equation}
\widehat{\mathcal{O}_{c,x}^{\alpha ,\theta }}(\xi )=e^{-c\psi _{\alpha
,\theta }(\xi )},\quad \alpha \in (0,2],\quad |\theta |\leq \min \{\alpha
,2-\alpha \}.  \label{cf}
\end{equation}

\begin{proof}
By exploiting the uniform convergence, the Fourier transform of (\ref{fr1})
can be evaluated as follows
\begin{eqnarray*}
\int_{-\infty }^{+\infty }e^{ix\xi }\mathcal{O}_{c,x}^{\alpha ,\theta
}f(x)dx &=&\sum_{n=0}^{\infty }\frac{c^{n}}{n!}\int_{-\infty }^{+\infty
}e^{ix\xi }g_{n}(x)dx \\
&=&[\text{by (\ref{uno})}]=-\psi _{\alpha ,\theta }(\xi )\sum_{n=0}^{\infty }%
\frac{c^{n}}{n!}\int_{-\infty }^{+\infty }e^{ix\xi }g_{n-1}(x)dx \\
&=&\sum_{n=0}^{\infty }\frac{(-c\psi _{\alpha ,\theta }(\xi ))^{n}}{n!}%
\widetilde{f}(\xi ).
\end{eqnarray*}
\end{proof}

\begin{remark}
In the integer order case $\alpha =1$ and for $\theta =-1,$ we obtain, by (%
\ref{fr1}) and (\ref{crr}), the shift operator%
\begin{equation*}
\mathcal{O}_{c,x}^{1,-1}f(x)=e^{-c\partial _{x}}f(x)=f(x-c),
\end{equation*}%
while, for $\alpha =2$ and $\theta =0$, we get%
\begin{equation*}
\mathcal{O}_{c,x}^{2,0}f(x)=e^{c\partial _{x}^{2}}f(x),
\end{equation*}%
which is a special case of the generalized exponential operator considered
in \cite{DATT}. Note that the symbol of $e^{c\partial _{x}^{2}}$ is $%
e^{-c\xi ^{2}}$.
\end{remark}
\end{lemma}
\end{definition}

The fractional shift operator has been introduced in \cite{BEGG}, in the
special case $\alpha \in (0,1)$ and $\theta =-\alpha .$

\

We define another pseudo-differential operator in the symmetric case, i.e.
for $\theta =0$. Under this assumption the symbol (\ref{due}) is real.

\begin{definition}
\textbf{(Fractional logarithmic operator) }Let $f\in L^{c}(\mathbb{R})$ be a
function s.t. $g_{j}f=\underbrace{\mathcal{D}_{x}^{\alpha }...\mathcal{D}%
_{x}^{\alpha }}_{j-times}f\in L^{c}(\mathbb{R})$, for any $j=0,...$. , then%
\begin{equation}
\mathcal{P}_{c,x}^{\alpha }f(x):=\sum_{n=1}^{\infty }\frac{c^{n}}{n}%
\underbrace{\mathcal{D}_{x}^{\alpha }...\mathcal{D}_{x}^{\alpha }}%
_{n-times}f(x),\qquad c>0,\quad \alpha \in (0,2].  \label{fr2}
\end{equation}%
provided that the series converges uniformly.

\begin{lemma}
The symbol of (\ref{fr2}) is given by%
\begin{equation}
\widehat{\mathcal{P}_{c,x}^{\alpha }}(\xi )=-\ln (1+c|\xi |^{\alpha }),\quad
c>0,\text{ }|\xi |<1/c^{1/\alpha }.  \label{sy2}
\end{equation}

\begin{proof}
As before, we get
\begin{eqnarray*}
\int_{-\infty }^{+\infty }e^{ix\xi }\mathcal{P}_{c,x}^{\alpha }f(x)dx
&=&\sum_{n=1}^{\infty }\frac{c^{n}}{n}\int_{-\infty }^{+\infty }e^{ix\xi }%
\underbrace{\mathcal{D}_{x}^{\alpha }...\mathcal{D}_{x}^{\alpha }}%
_{n-times}f(x)dx \\
&=&\sum_{n=1}^{\infty }\frac{(-c|\xi |^{\alpha })^{n}}{n}\widetilde{f}(\xi ),
\end{eqnarray*}%
which coincides with (\ref{sy2}).
\end{proof}
\end{lemma}
\end{definition}

In the integer order case $\alpha =2$, we can apply the semigroup property
of the standard derivatives, so that we have $\mathcal{P}_{c,x}^{2}f(x)=-\ln
\left( 1+c\frac{d^{2}}{dx^{2}}\right) f(x).$ The fractional logarithmic
operator has been used in \cite{BEG}, in connection with the geometric
stable processes.

\begin{remark}
We note that a sufficient condition for the uniform convergence of the
series in (\ref{fr1}) and (\ref{fr2}) is that $f(x)$ is an eigenfunction of
the RF fractional derivative. For example, when $\theta =-\alpha ,$ this is
the case for $f(x)=e^{-kx},$ for $x\in \mathbb{R},$ $k>0$. Indeed, in view
of (2.2.15) in \cite{KIL}, p.81, we have that%
\begin{equation*}
\mathcal{O}_{c,x}^{\alpha ,-\alpha }e^{-kx}=\sum_{n=0}^{\infty }\frac{%
(ck^{\alpha })^{n}}{n!}e^{-kx}=e^{-kx+ck^{\alpha }}.
\end{equation*}%
Moreover the uniform convergence of the series in (\ref{fr1}) and (\ref{fr2}%
) holds for the density of the $\alpha $-stable process, as the following
lemma shows.
\end{remark}

\begin{lemma}
The transition density $p_{\alpha ,\theta }(x;t)$ of the $\alpha $-stable
process $\mathcal{S}_{\alpha ,\theta }(t),t\geq 0$ satisfies the following
equations%
\begin{equation}
\mathcal{O}_{c,x}^{\alpha ,\theta }u(x,t)=e^{c\partial _{t}}u(x,t)=u(x,t+c).
\label{cc}
\end{equation}%
In the special case $\theta =0$, we also have that%
\begin{equation}
\mathcal{P}_{c,x}^{\alpha }u(x,t)=-\ln \left( 1-c\partial _{t}\right) u(x,t).
\label{cc2}
\end{equation}
\end{lemma}

\begin{proof}
Recall that $p_{\alpha ,\theta }(x;t)$ satisfies equation (\ref{mai}); thus
we get that%
\begin{eqnarray}
\mathcal{O}_{c,x}^{\alpha ,\theta }u(x,t) &=&\sum_{n=0}^{\infty }\frac{c^{n}%
}{n!}\partial _{t}^{n}u(x,t)  \label{ss} \\
&=&e^{c\partial _{t}}u(x,t),  \notag
\end{eqnarray}%
which coincides with (\ref{cc}). The first step can be checked by resorting
to the Fourier transform and considering (\ref{cf}):%
\begin{eqnarray*}
\mathcal{F}\left\{ \mathcal{O}_{c,x}^{\alpha ,\theta }u(x,t);\xi \right\}
&=&e^{-c\psi _{\alpha ,\theta }(\xi )}\widetilde{u}(\xi
,t)=\sum_{n=0}^{\infty }\frac{c^{n}}{n!}\partial _{t}^{n}e^{-\psi _{\alpha
,\theta }(\xi )t} \\
&=&[\text{by (\ref{due})}]=\mathcal{F}\left\{ \sum_{n=0}^{\infty }\frac{c^{n}%
}{n!}\partial _{t}^{n}u(x,t);\xi \right\} .
\end{eqnarray*}%
Analogously we obtain (\ref{cc2}) as follows:%
\begin{equation*}
\mathcal{P}_{c,x}^{\alpha }u(x,t)=\sum_{n=1}^{\infty }\frac{c^{n}}{n}%
\partial _{t}^{n}u(x,t),
\end{equation*}%
since, by (\ref{sy2}),%
\begin{eqnarray*}
\mathcal{F}\left\{ \mathcal{P}_{c,x}^{\alpha }u(x,t);\xi \right\} &=&-\ln
(1+c|\xi |^{\alpha })\widetilde{u}(\xi ,t) \\
&=&\sum_{n=1}^{\infty }\frac{(-c|\xi |^{\alpha })^{n}}{n}=\sum_{n=1}^{\infty
}\frac{c^{n}}{n}\partial _{t}^{n}e^{-|\xi |^{\alpha }t} \\
&=&\mathcal{F}\left\{ \sum_{n=1}^{\infty }\frac{c^{n}}{n}\partial
_{t}^{n}u(x,t);\xi \right\} .
\end{eqnarray*}
\end{proof}

\

We consider now an extension of the (space-)fractional diffusion equation (%
\ref{mai}), obtained by adding the fractional exponential operator. We prove
that, as a consequence, the stochastic process governed by the new equation
is again the stable process, but with a random time-argument represented by
the Poisson process with drift, i.e. $N(t)+at,$ $a,t\geq 0$. The latter has
been studied in \cite{BEG2}, where also the case of a general L\'{e}vy
process subordinated by it has been analyzed.

Let $N(t)$ be a homogeneous Poisson process with parameter $\lambda $,
independent from the stable process $\mathcal{S}_{\alpha ,\theta }$; we
define here the following subordinated process%
\begin{equation}
\mathcal{Z}_{\alpha }(t):=\mathcal{S}_{\alpha ,\theta }(at+N(t)),\qquad
t\geq 0,\text{ }a\geq 0  \label{zz}
\end{equation}%
by adding a random term in the time argument, which, in this case is
represented by the Poisson process. As special case, for $\lambda
\rightarrow 0,$ $a=1$, we obtain the standard stable process.

\begin{lemma}
\textbf{(Fractional diffusion-type equation with exponential differential
operator)} Let $f\in L^{c}(\mathbb{R})$ satisfy the condition given in
Def.1. Then the following initial-value problem
\begin{equation}
\left\{
\begin{array}{l}
\partial _{t}u(x,t)=\left[ a\mathcal{D}_{x}^{\alpha ,\theta }+\lambda (I-%
\mathcal{O}_{1,x}^{\alpha ,\theta })\right] u(x,t) \\
u(x,0)=f(x)%
\end{array}%
\right. ,  \label{cor}
\end{equation}%
is satisfied by the transition semigroup of the process $\mathcal{Z}_{\alpha
}$, given by%
\begin{equation}
\mathcal{T}_{t}^{\mathcal{Z}}f(x)=e^{-\lambda t}\sum_{k=0}^{\infty }\frac{%
(\lambda t)^{k}}{k!}\int_{\mathbb{R}}f(x-y)p_{\alpha ,\theta }(y,k+at)dy.
\label{cor2}
\end{equation}
\end{lemma}

\begin{proof}
By Lemma 7 it is easy to check that for the function (\ref{cor2}) the
fractional exponential operator $\mathcal{O}_{c,x}^{\alpha ,\theta }$ is
well defined$.$ Taking the Fourier transform of (\ref{cor}), we obtain%
\begin{equation}
\partial _{t}\widetilde{u}(\xi ,t)=-\left[ a\psi _{\alpha ,\theta }(\xi
)+\lambda (1-e^{-\psi _{\alpha ,\theta }(\xi )})\right] \widetilde{u}(\xi
,t),  \label{ast}
\end{equation}%
by (\ref{uno}) and (\ref{cf}). On the other hand, from (\ref{cor2}), we have
that
\begin{eqnarray*}
\widetilde{u}(\xi ,t) &=&\widetilde{f}(\xi )e^{-\lambda t}\sum_{k=0}^{\infty
}\frac{(\lambda t)^{k}}{k!}\widetilde{p}_{\alpha ,\theta }(\xi ,k+at) \\
&=&\widetilde{f}(\xi )e^{-\lambda t}\sum_{k=0}^{\infty }\frac{(\lambda t)^{k}%
}{k!}e^{-\psi _{\alpha ,\theta }(\xi )(k+at)} \\
&=&\widetilde{f}(\xi )\exp \{-\left[ \lambda (1-e^{-\psi _{\alpha ,\theta
}(\xi )})+a\psi _{\alpha ,\theta }(\xi )\right] t\},
\end{eqnarray*}%
which, differentiated w.r.t. $t$, gives (\ref{ast}).
\end{proof}

\begin{remark}
It is easy to see that, for $a\geq 0$, the time argument in (\ref{zz}) is a
subordinator and thus this time change represents a subordination. Moreover $%
\mathcal{Z}_{\alpha }$ is\ a L\'{e}vy process, since it is given by the
composition of two independent L\'{e}vy processes. Its L\'{e}vy symbol can
be obtained as follows, by considering the Laplace transform $\mathbb{E}%
e^{-s(N(t)+at)}=\exp \{-sat-\lambda t(1-e^{-s})\}$:
\begin{eqnarray}
\eta _{\mathcal{Z}_{\alpha }}(\xi ) &=&\frac{1}{t}\ln \left\{ \mathbb{E}%
\left( \left. \mathbb{E}e^{i\xi \mathcal{S}_{\alpha ,\theta
}(at+N(t))}\right\vert N(t)\right) \right\} =\frac{1}{t}\ln \left\{ \mathbb{E%
}e^{-\psi _{\alpha ,\theta }(\xi )[at+N(t)]}\right\}  \notag \\
&=&\frac{1}{t}\ln \left\{ e^{-a\psi _{\alpha ,\theta }(\xi )t-\lambda
t(1-e^{-\psi _{\alpha ,\theta }(\xi )})}\right\} =-a\psi _{\alpha ,\theta
}(\xi )-\lambda (1-e^{-\psi _{\alpha ,\theta }(\xi )}).  \notag
\end{eqnarray}%
The L\'{e}vy measure can be evaluated, by applying Theorem 30.1, p.197\ in
\cite{SATO}, as follows%
\begin{eqnarray*}
\nu _{\mathcal{Z}_{\alpha }}(x) &=&a\nu _{\mathcal{S}_{\alpha ,\theta
}}(x)+\lambda \int_{0}^{+\infty }p_{\alpha ,\theta }(x,s)\delta (s-1)ds \\
&=&a\left[ \frac{P}{x^{1+\alpha }}1_{(0,+\infty )}(x)+\frac{Q}{|x|^{1+\alpha
}}1_{(-\infty ,0)}(x)\right] +\lambda p_{\alpha ,\theta }(x,1),
\end{eqnarray*}%
by considering that the drift coefficient of the random time argument is
equal to $a$. For $\alpha \in (0,2)$ the diffusion coefficient is $A_{%
\mathcal{Z}_{\alpha }}=0$ and thus the process is a pure jump process.
Moreover, for $\alpha \in (0,1),$ the process has finite variation, since $%
\int_{|x|\leq 1}|x|\nu _{\mathcal{Z}_{\alpha }}(dx)<\infty $. On the other
hand, since $\nu _{\mathcal{Z}_{\alpha }}(\mathbb{R})=\infty $ for any $%
\alpha $, the expected number of jumps in any finite interval is infinite,
i.e. the process displays infinite activity. The drift coefficient reads%
\begin{eqnarray*}
\gamma _{\mathcal{Z}_{\alpha }} &=&a\int_{0}^{+\infty }\delta
(s-1)ds\int_{|x|\leq 1}xp_{\alpha ,\theta }^{s}(x;1)dx \\
&=&a\int_{|x|\leq 1}xp_{\alpha ,\theta }(x;1)dx.
\end{eqnarray*}%
For $\alpha =2$ and $\theta =0,$ we obtain the L\'{e}vy triplet of the
process $\mathcal{Z}_{2}(t)=W(at+N(t)),t\geq 0$ (where $W(t),t>0$ is a
standard Brownian motion with characteristic function $e^{-\xi ^{2}t}$ and
thus variance equal to $2t$):%
\begin{equation}
\nu _{\mathcal{Z}_{2}}(x)=\lambda \int_{0}^{+\infty }\frac{e^{-x^{2}/2s}}{%
\sqrt{2\pi s}}\delta (s-1)ds=\lambda \frac{e^{-x^{2}/2}}{\sqrt{2\pi }},
\end{equation}%
by considering that the L\'{e}vy measure of the Brownian motion is zero, and%
\begin{equation}
A_{\mathcal{Z}_{2}}=a,\qquad \gamma _{\mathcal{Z}_{2}}=0.
\end{equation}%
The first two moments exist and can be obtained by deriving the
characteristic function, which reads $\mathbb{E}e^{i\xi \mathcal{Z}%
_{2}(t)}=\exp \{-\xi ^{2}at-\lambda (1-e^{-\xi ^{2}})t\}$:%
\begin{eqnarray*}
\mathbb{E}\mathcal{Z}_{2}(t) &=&0 \\
Var\mathcal{Z}_{2}(t) &=&2t\left( a+\lambda \right) .
\end{eqnarray*}
\end{remark}

\begin{remark}
As a consequence of Lemma 8, we can write the generator of $\mathcal{Z}%
_{\alpha }$ as
\begin{equation}
\mathcal{G}_{\alpha }f(x)=a\mathcal{D}_{x}^{\alpha ,\theta }f(x)-\lambda
\int_{\mathbb{R}}\left( f(x+y)-f(x)\right) p_{\alpha ,\theta }(y,1)dy
\label{ff}
\end{equation}%
with symbol $\widehat{\mathcal{G}_{\alpha }}(\xi )=-a\psi _{\alpha ,\theta
}(\xi )+\lambda (1-e^{-\psi _{\alpha ,\theta }(\xi )}).$ Equation (\ref{ff})
can be alternatively obtained from (5.2) in \cite{BEG2}, by considering that
the generator of the stable process is the RF derivative and taking the
fractional parameters in \cite{BEG2} equal to one. Moreover it suggests the
following alternative expression of the fractional shift operator as
integral transform of the standard shift operator, i.e.%
\begin{equation*}
\mathcal{O}_{c,x}^{\alpha ,\theta }f(x)=\int_{\mathbb{R}}e^{cy\partial
_{x}}p_{\alpha ,\theta }(y,1)f(x)dy,
\end{equation*}%
which, in the case $\alpha \in (0,1)$ and for $c=-1,$ coincides with the
definition given in \cite{DOV}.
\end{remark}

\begin{remark}
In the symmetric case and for $\alpha =2$, the previous result shows that
the density of the process $\mathcal{Z}_{2}(t)=W(at+N(t))$, $t\geq 0$
satisfies the following equation%
\begin{equation*}
\partial _{t}u(x,t)=\left[ a\partial _{x}^{2}+\lambda (I-e^{-\partial
_{x}^{2}})\right] u(x,t),
\end{equation*}%
with initial condition $u(x,0)=\delta (x).$ On the other hand, for $\alpha
=1/2$ and $\theta =-1/2$, we have another interesting special case given by $%
\mathcal{Z}_{1/2}(t)=\mathcal{S}_{1/2}(at+N(t)),$ $t\geq 0$, where $\mathcal{%
S}_{1/2}(t)$ is the L\'{e}vy subordinator. The latter is known to be equal
in distribution to the first passage time of a Brownian motion through the
level $t$, i,e, $T_{t}:=\inf \{s>0:W(s)\geq t\}.$ Therefore we have the
following equality in the sense of the finite-dimensional distributions
(i.d.) with the first passage time of a Brownian motion through the
trajectories of the process $N(t)+at,$ i.e.%
\begin{equation}
\mathcal{Z}_{1/2}(t)\overset{i.d.}{=}T_{t}^{a}=\inf \{s>0:W(s)\geq N(t)+at\}.
\label{ts}
\end{equation}%
The transition density of (\ref{ts}) is thus the solution to the equation%
\begin{equation*}
\partial _{t}u(x,t)=\left[ a\partial _{x}^{1/2}+\lambda (I-\mathcal{O}%
_{1,x}^{1/2,-1/2})\right] u(x,t),
\end{equation*}%
with initial condition $u(x,0)=\delta (x).$
\end{remark}

\section{Main results}

\subsection{Fractional diffusion-type equation with logarithmic differential
operator}

We study now the extension of the fractional diffusion equation obtained by
adding the logarithmic differential operator $\mathcal{P}_{c,x}^{\alpha }$
to (\ref{mai}), in the symmetric case (i.e. for $\theta =0$)$.$\ In analogy
with the previous results we show that this additional term introduces a
random element in the time argument of the corresponding stable process. In
this case, instead of the Poisson process, we have a gamma process. We
denote by $\Gamma (t),t\geq 0$ the gamma subordinator of parameters $\mu
,\rho >0,$ i.e. with density
\begin{equation}
f_{\Gamma }(x,t):=\Pr \left\{ \Gamma (t)\in dx\right\} /dx=\left\{
\begin{array}{l}
\frac{\rho ^{\mu t}}{\Gamma (\mu t)}x^{\mu t-1}e^{-\rho x},\qquad x\geq 0 \\
0,\qquad x<0%
\end{array}%
\right. .  \label{gam}
\end{equation}%
Note that, for $\mu =0$ the process $\Gamma (t)$ reduces to the elementary
subordinator $t.$

\begin{theorem}
Let $f\in L^{c}(\mathbb{R})$ satisfy the conditions given in Def.4, then the
solution to the following initial-value problem%
\begin{equation}
\left\{
\begin{array}{l}
\partial _{t}u(x,t)=\left[ a\mathcal{D}_{x}^{\alpha }+\mu \mathcal{P}%
_{1/\rho ,x}^{\alpha }\right] u(x,t) \\
u(x,0)=f(x).%
\end{array}%
\right. ,\quad x\in \mathbb{R},\text{ }t>0,\,\alpha \in (0,2],  \label{pr3}
\end{equation}%
coincides with the semigroup $\mathcal{T}_{t}^{\mathcal{X}}f(x)=\mathbb{E}f%
\left[ x-\mathcal{X}_{\alpha }(t)\right] $ of the following subordinated
process%
\begin{equation}
\mathcal{X}_{\alpha }(t)=\mathcal{S}_{\alpha }(at+\Gamma (t)),\quad t\geq 0,
\label{pr}
\end{equation}%
where $\mathcal{S}_{\alpha }(t),t\geq 0,$ is a symmetric stable process
defined in (\ref{ch2}) for $\theta =0$, with density $p_{\alpha }(x;t)$, and
$\Gamma (t),t\geq 0$ is an independent gamma subordinator with density (\ref%
{gam}).
\end{theorem}

\begin{proof}
If we take the Fourier transform of the first equation in (\ref{pr3}) we
get, in view of (\ref{uno}) together with Lemma 5,%
\begin{equation*}
\partial _{t}\widetilde{u}(\xi ,t)=-\left[ a|\xi |^{\alpha }+\mu \ln (1+|\xi
|^{\alpha }/\rho )\right] \widetilde{u}(\xi ,t).
\end{equation*}%
Now we evaluate the characteristic function of (\ref{pr}), by considering
that $\mathbb{E}e^{-s\Gamma (t)}:=e^{-\Psi _{\Gamma }(s)t}=1/\left( 1+s/\rho
\right) ^{\mu t}$:

\begin{eqnarray*}
\mathbb{E}e^{i\xi \mathcal{X}_{\alpha }(t)} &=&\mathbb{E}\left[ \mathbb{E}%
\left( \left. e^{i\xi \mathcal{S}_{\alpha }(at+\Gamma (t))}\right\vert
\Gamma (t)\right) \right] =e^{-a|\xi |^{\alpha }t}\mathbb{E}e^{-|\xi
|^{\alpha }\Gamma (t)} \\
&=&\frac{e^{-a|\xi |^{\alpha }t}}{\left( 1+|\xi |^{\alpha }/\rho \right)
^{\mu t}}.
\end{eqnarray*}%
The time argument in (\ref{pr}) is represented by the gamma process with
drift, which is a L\'{e}vy process and also a subordinator, being strictly
increasing a.s. Then the process (\ref{pr}) is\ itself a L\'{e}vy process
and its L\'{e}vy symbol is%
\begin{eqnarray}
\eta _{\mathcal{X}_{\alpha }}(\xi ) &=&\frac{1}{t}\ln \left\{ \mathbb{E}%
e^{i\xi \mathcal{X}_{\alpha }(t)}\right\}  \label{lev5} \\
&=&-a|\xi |^{\alpha }-\mu \ln \left( 1+\frac{|\xi |^{\alpha }}{\rho }\right)
.  \notag
\end{eqnarray}%
Thus%
\begin{equation*}
\partial _{t}\widetilde{u}(\xi ,t)=-\eta _{\mathcal{X}_{\alpha }}(\xi )%
\widetilde{u}(\xi ,t)
\end{equation*}%
so that%
\begin{equation*}
u(x,t)=\mathcal{F}^{-1}\left\{ e^{-\eta _{\mathcal{X}_{\alpha }}(\xi )t}%
\widetilde{f}(\xi );x\right\} =\mathcal{T}_{t}^{\mathcal{X}}f(x)
\end{equation*}%
is the solution of (\ref{pr3}), where $\mathcal{F}^{-1}$ denotes the inverse
Fourier transform.
\end{proof}

\begin{remark}
The previous results are particularly interesting in the special case $%
\alpha =2$, since they imply that the solution to the p.d.e.%
\begin{eqnarray}
\partial _{t}u(x,t) &=&\left[ a\partial _{x}^{2}+\mu \mathcal{P}_{1/\rho
,x}^{2}\right] u(x,t)  \label{vg} \\
&=&\left[ a\partial _{x}^{2}-\mu \ln \left( 1+\frac{\partial _{x}^{2}}{\rho }%
\right) \right] u(x,t),  \notag
\end{eqnarray}%
coincides with the transition density of the process $W(at+\Gamma (t)),$ $%
t>0 $. For $a=0$, equation (\ref{vg}) provides the generator of the variance
gamma process, which can be explicitly written as%
\begin{equation*}
\mathcal{A}=-\ln \left( 1+\frac{\partial _{x}^{2}}{\rho }\right) ,
\end{equation*}%
by exploiting the semigroup property of the integer-order derivatives.
\end{remark}

It is evident from (\ref{lev5}) that $\mathcal{X}_{\alpha }$ can be
considered as a generalization of both stable and geometric stable processes
(see, for example, \cite{KOZ2}), to which it reduces in the special cases $%
\mu =0$ and $a=0,$ respectively. Again, by applying Theorem 30.1, p.197\ in
\cite{SATO}, we get, for $\alpha \in (0,2),$ the L\'{e}vy triplet:%
\begin{equation}
\nu _{\mathcal{X}_{\alpha }}(\cdot )=a\,\nu _{\mathcal{S}_{\alpha }}(\cdot
)+\mu \int_{0}^{+\infty }s^{-1}e^{-\rho s}p_{\alpha }(\cdot ;s)ds,
\label{lev}
\end{equation}%
\begin{equation*}
A_{\mathcal{X}_{\alpha }}=0
\end{equation*}%
and%
\begin{equation*}
\gamma _{\mathcal{X}_{\alpha }}=\mu \int_{0}^{+\infty }s^{-1}e^{-\rho
s}ds\int_{|x|\leq 1}xp_{\alpha }(x;s)dx=0,
\end{equation*}%
since the stable process is symmetric by assumption. From (\ref{lev}) we can
deduce that the asymptotic behavior of the L\'{e}vy measure at the origin,
for any positive $a$, is polynomial, as for the stable processes, while, for
the geometric stable, it is logarithmic (see \cite{KOZ3}).

For $\alpha =2,$ the L\'{e}vy measure of $\mathcal{X}_{2}(t)=W(at+\Gamma
(t)) $ is given instead by
\begin{equation}
\nu _{\mathcal{X}_{2}}(s)=\mu \int_{0}^{+\infty }\frac{e^{-s^{2}/2z}}{\sqrt{%
2\pi z^{3}}}e^{-\rho z}dz=\frac{\mu }{|s|}e^{-\sqrt{2\rho }|s|},
\label{lev3}
\end{equation}%
and the diffusion and drift parameters are respectively equal to%
\begin{equation}
A_{\mathcal{X}_{2}}=1,\qquad \gamma _{\mathcal{X}_{2}}=0.  \label{lev4}
\end{equation}%
We can compare (\ref{lev3}) and (\ref{lev4}) to the L\'{e}vy triplet of the
symmetric VG process $W(\Gamma (t)),t\geq 0$ (which corresponds to the
special case $a=0$): the L\'{e}vy measure is the same, but in the VG case
the diffusion coefficient is equal to zero (i.e. $A=0)$ and the process is a
pure jump process with infinitely many jumps and finite variation, since $%
\int_{|x|\leq 1}|x|\nu (dx)<\infty $. On the other hand here we have, in
view of Theorem 21.9. in \cite{SATO}, an infinite variation of almost all
paths of $\mathcal{X}_{2},$ since $A_{\mathcal{X}_{2}}\neq 0$, and thus it
is not a pure jump process. Jump-diffusion models are extensions of pure
jumps models, mixing a jump process and a diffusion process, particularly
useful in option pricing (see e.g. \cite{CON}).

For $\alpha =2,$ the first two moments exist and can be obtained by deriving
the characteristic function, which reads $\mathbb{E}e^{-i\xi \mathcal{X}%
_{2}(t)}=e^{-a\xi ^{2}t-\mu t\ln (1+\xi ^{2}/\rho )}$:%
\begin{eqnarray*}
\mathbb{E}\mathcal{X}_{2}(t) &=&0 \\
Var\mathcal{X}_{2}(t) &=&2t\left( a+\frac{\mu }{\rho }\right) .
\end{eqnarray*}%
In the general case, for $\alpha \in (0,2)$ we can only evaluate the
fractional moment of order $\gamma \in (-1,\alpha )$, by applying Theorem 3
in \cite{SHA},
\begin{eqnarray}
\mathbb{E}\mathcal{X}_{\alpha }^{\gamma }(t) &=&\int_{t}^{+\infty }\mathbb{E}%
\mathcal{S}_{\alpha }^{\gamma }(s)f_{\Gamma }(s-at,t)ds  \label{ser} \\
&=&\frac{2^{\gamma }\Gamma \left( 1-\frac{\gamma }{\alpha }\right) \Gamma
\left( \frac{1+\gamma }{2}\right) }{\sqrt{\pi }\Gamma \left( 1-\frac{\gamma
}{2}\right) }\int_{0}^{+\infty }\mathbb{(}at+s\mathbb{)}^{\gamma /\alpha
}f_{\Gamma }(s,t)ds  \notag \\
&=&\frac{2^{\gamma }\Gamma \left( 1-\frac{\gamma }{\alpha }\right) \Gamma
\left( \frac{1+\gamma }{2}\right) }{\sqrt{\pi }\Gamma \left( 1-\frac{\gamma
}{2}\right) }\sum_{j=0}^{\infty }\binom{\gamma /\alpha }{j}%
(at)^{j}\int_{0}^{+\infty }s^{\gamma /\alpha -j}f_{\Gamma }(s,t)ds  \notag \\
&=&\frac{\gamma 2^{\gamma }\Gamma \left( 1-\frac{\gamma }{\alpha }\right)
\Gamma \left( \frac{1+\gamma }{2}\right) \Gamma \left( \frac{\gamma }{\alpha
}\right) }{\alpha \sqrt{\pi }\Gamma \left( 1-\frac{\gamma }{2}\right) \Gamma
(\mu t)\rho ^{\gamma /\alpha }}\sum_{j=0}^{\infty }\frac{(\rho at)^{j}}{j!}%
\frac{\Gamma \left( \frac{\gamma }{\alpha }+\mu t-j\right) }{\Gamma \left(
\frac{\gamma }{\alpha }+1-j\right) }  \notag \\
&=&\frac{\gamma 2^{\gamma }\sqrt{\pi }2^{1-\gamma }\Gamma (\gamma )}{\alpha
\sin \left( \pi \gamma /\alpha \right) \Gamma \left( 1-\frac{\gamma }{2}%
\right) \Gamma \left( \frac{\gamma }{2}\right) \Gamma (\mu t)\rho ^{\gamma
/\alpha }}\,_{1}\Psi _{1}\left( \left. \rho at\right\vert \mycom{(\gamma
/\alpha +\mu t,-1)}{(\gamma /\alpha +1,-1)}\right)  \notag \\
&=&\frac{2\sin \left( \pi \gamma /2\right) \Gamma (\gamma +1)}{\alpha \sqrt{%
\pi }\sin \left( \pi \gamma /\alpha \right) \Gamma (\mu t)\rho ^{\gamma
/\alpha }}\,_{1}\Psi _{1}\left( \left. \rho at\right\vert \mycom{(\gamma
/\alpha +\mu t,-1)}{(\gamma /\alpha +1,-1)}\right) ,  \notag
\end{eqnarray}%
where $_{1}\Psi _{1}$ denotes the generalized Wright function with $p=q=1$
(see \cite{KIL}, p.56). By applying Theorem 1.5 in \cite{KIL}, p.58, it is
easy to check that the series in (\ref{ser}) is absolutely convergent for
all $t.$

Finally, we show that the tails' behavior of the density of $\mathcal{X}%
_{\alpha }(t),$ for any fixed $t,$ is the same (up to a different constant)
of those holding for both the stable and geometric stable random variables
(see \cite{SAMO}, p. 17, and \cite{KOZ}, respectively).

\begin{theorem}
For $\alpha \in (0,2)$, we have that%
\begin{equation}
\left\{
\begin{array}{c}
\lim_{x\rightarrow \infty }x^{\alpha }P(\mathcal{X}_{\alpha }(t)>x)=\frac{%
C_{\alpha ,\theta }\left( a+\frac{\mu }{\rho }\right) t}{\Gamma (1-\alpha )}
\\
\lim_{x\rightarrow \infty }x^{\alpha }P(\mathcal{X}_{\alpha }(t)<-x)=\frac{%
C_{\alpha ,\theta }^{\prime }\left( a+\frac{\mu }{\rho }\right) t}{\Gamma
(1-\alpha )}%
\end{array}%
\right. ,\qquad t\geq 0,  \label{tau}
\end{equation}%
where $C_{\alpha ,\theta }=\frac{1}{2}\left[ 1-\frac{\tan (\pi \alpha /2)}{%
\tan (\pi \theta /2)}\right] $ and $C_{\alpha ,\theta }^{\prime }=\frac{1}{2}%
\left[ 1+\frac{\tan (\pi \alpha /2)}{\tan (\pi \theta /2)}\right] $
\end{theorem}

\begin{proof}
We start by considering the process $\mathcal{S}_{\alpha ,\theta }(at+\Gamma
(t)),$ $t\geq 0$, in the special case $\alpha \in (0,1),$ $\theta =-\alpha $%
, for which we can write that%
\begin{eqnarray*}
\int_{0}^{+\infty }e^{-\eta x}P(\mathcal{S}_{\alpha ,-\alpha }(at+\Gamma
(t))>x)dx &=&\frac{1-\mathbb{E}e^{-\eta \mathcal{S}_{\alpha ,-\alpha
}(at+\Gamma (t))}}{\eta } \\
&=&\frac{1-\exp \{-a\eta ^{\alpha }t\}-\left( 1+\frac{\eta ^{\alpha }}{\rho }%
\right) ^{\mu t}}{\eta } \\
&\sim &\left( at+\frac{\mu t}{\rho }\right) \eta ^{\alpha -1},
\end{eqnarray*}%
for any fixed $t$ and for $\eta \rightarrow 0.$ By applying the Tauberian
theorem (see Theorem XIII-5-4, p.446, in \cite{FEL}), we obtain the first
equation in (\ref{tau}), with $C_{\alpha ,\theta }=1$. The case $\alpha \in
(0,1),$ $|\theta |\leq \alpha $ can be obtained by equation (1.2.6) in \cite%
{SAMO}, while for $\alpha \geq 1$ and $\theta =0$ we adapt Proposition 1.3.1
in \cite{SAMO}, p.20, to the r.v. $\mathcal{S}_{\alpha }(at+\Gamma (t))$,
for fixed $t$. Let $A_{\alpha /\alpha ^{\prime }}$ be a (totally skewed to
the right) stable r.v. of index $\alpha /\alpha ^{\prime },$ with $\alpha
^{\prime }>\alpha $ and $\theta =-\alpha /\alpha ^{\prime },\sigma ^{\prime
} $ =$\left( \cos \frac{\pi \alpha }{2\alpha ^{\prime }}\right) ^{\alpha
^{\prime }/\alpha }$ with Laplace transform
\begin{equation*}
\mathbb{E}e^{-sA_{\alpha /\alpha ^{\prime }}}=\exp \left\{ -s^{\alpha
/\alpha ^{\prime }}\right\}
\end{equation*}%
and let $X_{\alpha ^{\prime }}^{\Gamma }$ be a (symmetric) stable r.v. of
index $\alpha ^{\prime },$ $\theta =0$ and with
\begin{equation*}
\mathbb{E}\left\{ \left. e^{i\xi X_{\alpha ^{\prime }}^{\Gamma }}\right\vert
\Gamma \right\} =\exp \{-|\xi |^{\alpha ^{\prime }}\left[ at+\Gamma (t)%
\right] ^{\alpha ^{\prime }/\alpha }\}.
\end{equation*}%
Thus we prove that
\begin{equation*}
Z_{\alpha }:=\left( A_{\alpha /\alpha ^{\prime }}\right) ^{1/\alpha ^{\prime
}}X_{\alpha ^{\prime }}
\end{equation*}%
is a stable r.v. of index $\alpha $, with parameters $\theta =0$: indeed we
can write%
\begin{eqnarray*}
\mathbb{E}e^{i\xi Z_{\alpha }} &=&\mathbb{E}e^{i\xi A_{\alpha /\alpha
^{\prime }}^{1/\alpha ^{\prime }}X_{\alpha ^{\prime }}}=\mathbb{E}\left\{
\mathbb{E}\left[ \left. e^{i\xi \left( A_{\alpha /\alpha ^{\prime }}\right)
^{1/\alpha ^{\prime }}X_{\alpha ^{\prime }}^{\Gamma }}\right\vert A_{\alpha
/\alpha ^{\prime }}\right] \right\} \\
&=&\mathbb{E}\left\{ \mathbb{E}\left[ \left. \mathbb{E}\left( \left. e^{i\xi
\left( A_{\alpha /\alpha ^{\prime }}\right) ^{1/\alpha ^{\prime }}X_{\alpha
^{\prime }}}\right\vert \Gamma \right) \right\vert A_{\alpha /\alpha
^{\prime }}\right] \right\} \\
&=&\mathbb{E}\left\{ \mathbb{E}\left[ \left. e^{-|\xi |^{\alpha ^{\prime }}%
\left[ at+\Gamma (t)\right] ^{\alpha ^{\prime }/\alpha }A_{\alpha /\alpha
^{\prime }}}\right\vert A_{\alpha /\alpha ^{\prime }}\right] \right\} \\
&=&\mathbb{E}e^{-|\xi |^{\alpha }\left[ at+\Gamma (t)\right] },
\end{eqnarray*}%
which is the characteristic function of the r.v. $\mathcal{S}_{\alpha
}(at+\Gamma (t))$, for fixed $t.$
\end{proof}

\subsection{Nonlinear fractional diffusion-type equation}

We consider now a non-linear extension of the equation (\ref{pp}), defined
as follows%
\begin{equation}
\partial _{t}u(x,t)=\left[ a\mathcal{D}_{x}^{\alpha ,\theta }+\lambda
at\left( I-\mathcal{O}_{1,x}^{\alpha ,\theta }\right) \right] u(x,t)-\lambda
\left( I-\mathcal{O}_{1,x}^{\alpha ,\theta }\right) I_{\theta }^{\alpha -1}%
\left[ xu(x,t)\right] ,  \label{ext}
\end{equation}%
and we prove that, for $\alpha \in (1,2),$ and under the initial condition $%
u(x,0)=p_{\alpha }^{\theta }(x,1),$ the solution coincides with the
transition density of the stable process $\mathcal{S}_{\alpha }^{\theta }$
time-changed by the so-called linear birth process with drift.

The form of the above equation is suggested by the following preliminary
result. Let $B(t)$ $t\geq 0,$ be a linear birth (Yule-Furry) process with
one progenitor and parameter $\lambda >0$. We recall that it is a Markov and
(a.s.) non-decreasing process, with one-dimensional distribution
\begin{equation*}
q_{k}(t):=\Pr \{\left. B(t)=k\right\vert B(0)=1\}=e^{-\lambda
t}(1-e^{-\lambda t})^{k-1},\quad k=1,2,...
\end{equation*}%
which is solution to the initial-value problem%
\begin{equation}
\frac{d}{dt}q_{k}(t)=-\lambda kq_{k}(t)+\lambda (k-1)q_{k-1}(t),\quad
q_{k}(0)=1_{k=1}.  \label{q3}
\end{equation}

\begin{lemma}
The density $q_{a}(x,t)$ of the linear birth process with positive drift,
defined as $B(t)+at,$ $a,t\geq 0$, satisfies the following equation:%
\begin{equation}
\partial _{t}u(x,t)=a\left[ \lambda t\left( I-e^{-\partial _{x}}\right) -%
\mathcal{\partial }_{x}\right] u(x,t)-\lambda \left( I-e^{-\partial
_{x}}\right) \left[ xu(x,t)\right] ,\qquad x\geq at+1,\text{ }t\geq 0,
\label{pre1}
\end{equation}%
with initial condition $u(x,0)=\delta (x-1).$
\end{lemma}

\begin{proof}
It is easy to check that the characteristic function of $B(t)+at$ is equal to%
\begin{equation}
\Phi _{B(t)+at}(\xi ):=\mathbb{E}e^{i\xi \lbrack B(t)+at]}=\frac{e^{-\lambda
t+i\xi +i\xi at}}{1-(1-e^{-\lambda t})e^{i\xi }},  \label{cff}
\end{equation}%
so that
\begin{eqnarray}
\partial _{t}\Phi _{B(t)+at}(\xi ) &=&(-\lambda +i\xi a)\Phi _{B(t)+at}(\xi
)+\lambda \frac{e^{-\lambda t+i\xi }}{1-(1-e^{-\lambda t})e^{i\xi }}\Phi
_{B(t)+at}(\xi )  \label{pre2} \\
&=&i\xi a\Phi _{B(t)+at}(\xi )-\lambda \frac{1-(1-e^{-\lambda t})e^{i\xi
}-e^{-\lambda t+i\xi }}{1-(1-e^{-\lambda t})e^{i\xi }}\Phi _{B(t)+at}(\xi )
\notag \\
&=&i\xi a\Phi _{B(t)+at}(\xi )-\lambda (1-e^{i\xi })\frac{\Phi
_{B(t)+at}(\xi )}{1-(1-e^{-\lambda t})e^{i\xi }}.  \notag
\end{eqnarray}%
We now concentrate to the last fraction and we consider the following fact%
\begin{eqnarray}
\mathcal{F}\{xq_{a}(x,t);\xi \} &=&e^{-\lambda t}\sum_{k=1}^{\infty
}(1-e^{-\lambda t})^{k-1}\int_{-\infty }^{+\infty }e^{ix\xi }x\delta
(x-k-at)dx  \label{dd} \\
&=&e^{-\lambda t}\sum_{k=1}^{\infty }(1-e^{-\lambda t})^{k-1}e^{i\xi
(k+at)}(k+at)  \notag \\
&=&at\Phi _{B(t)+at}(\xi )+\frac{\Phi _{B(t)+at}(\xi )}{1-(1-e^{-\lambda
t})e^{i\xi }},  \notag
\end{eqnarray}%
so that we can rewrite (\ref{pre2}) as%
\begin{equation}
\partial _{t}\Phi _{B(t)+at}(\xi )=i\xi a\Phi _{B(t)+at}(\xi )-\lambda
(1-e^{i\xi })\left[ \mathcal{F}\{xq_{a}(x,t);\xi \}-at\Phi _{B(t)+at}(\xi )%
\right] .  \notag
\end{equation}%
The last equation coincides with the Fourier transform of (\ref{pre1}). The
initial condition is satisfied, as can be checked by considering (\ref{cff})
for $t=0.$
\end{proof}

\begin{remark}
We note that equation (\ref{pre1}), for $a\neq 0,$ is an extension to the
continuous domain of the equation (\ref{q3}) governing the usual birth
process, to which it reduces in the limit, for $a\rightarrow 0.$ In the last
limiting case the shift operator $e^{-\partial _{x}}$ is replaced by the
backward difference operator $\Delta $, defined as $\Delta f(k)=f(k-1).$
\end{remark}

The linear and non-linear birth processes have been treated in the
fractional case (by considering a fractional time-derivative in (\ref{q3}))
by \cite{ORS} and \cite{ORS2}; see also \cite{ALI}. In the last reference
the birth process subordinated by an independent stable subordinator (i.e. $%
B(\mathcal{S}_{\alpha }^{\theta }(t))$, for $\theta =-\alpha $) is
considered.

We now define the following process

\begin{equation}
\mathcal{Y}_{\alpha }^{\theta }(t):=\mathcal{S}_{\alpha }^{\theta
}(at+B(t)),\qquad t>0,  \label{sub}
\end{equation}%
where $B$ is independent of $\mathcal{S}_{\alpha }^{\theta }.$ We recall
that (\ref{sub}) cannot be indicated as a `subordinated' process, since $B$
is not a subordinator, but we will refer to it as a randomly `time-changed'
process. For an overview on time change, see \cite{VERW} and the references
therein. The characteristic function of $\mathcal{Y}_{\alpha }^{\theta
}(t),t\geq 0$, can be evaluated as follows%
\begin{eqnarray}
\Phi _{\mathcal{Y}_{\alpha }^{\theta }(t)}(\xi ) &:&=\mathbb{E}e^{i\xi
\mathcal{Y}_{\alpha }^{\theta }(t)}=\mathbb{E}\left( \left. \mathbb{E}%
e^{i\xi \mathcal{S}_{\alpha }^{\theta }(at+B(t))}\right\vert B(t)\right)
\label{ch} \\
&=&e^{-a\psi _{\alpha ,\theta }(\xi )t}\mathbb{E}e^{-|\psi _{\alpha ,\theta
}(\xi )B(t)}=\frac{e^{-\lambda t-at\psi _{\alpha ,\theta }(\xi )-\psi
_{\alpha ,\theta }(\xi )}}{1-(1-e^{-\lambda t})e^{-\psi _{\alpha ,\theta
}(\xi )}},  \notag
\end{eqnarray}%
by recalling the Laplace transform of the Yule-Furry process, i.e.%
\begin{equation*}
\mathbb{E}e^{-sB(t)}=\frac{e^{-s-\lambda t}}{1-e^{-s}(1-e^{-\lambda t})}%
,\qquad \mathcal{R}e(s)>0
\end{equation*}%
and considering that $\mathcal{R}e(\psi _{\alpha ,\theta }(\xi ))=\mathcal{R}%
e(|\xi |^{\alpha }e^{i\,sign(\xi )\theta \pi /2})=|\xi |^{\alpha }\cos
(\theta \pi /2)\geq 0,$ for $|\theta |\leq \min \{\alpha ,2-\alpha \}$ and $%
\alpha \in (0,2].$

As it is evident from (\ref{ch}), the process $\mathcal{Y}_{\alpha }^{\theta
}(t),$ $t\geq 0$ is not L\'{e}vy, even though it is still Markov.

We now restrict our analysis to the case $\alpha \in (1,2],$ so that the
following holds
\begin{equation}
\left\vert \int_{-\infty }^{+\infty }e^{i\xi x}xp_{\alpha }^{\theta
}(x,t)dx\right\vert \leq \int_{-\infty }^{+\infty }\left\vert x\right\vert
p_{\alpha }^{\theta }(x,t)dx<\infty  \label{dom}
\end{equation}

\begin{lemma}
Let $\mathcal{F}\{xu(x,t);\xi \}:=\mathbb{E}\left( \mathcal{Y}_{\alpha
}^{\theta }(t)e^{i\xi \mathcal{Y}_{\alpha }^{\theta }(t)}\right) .$ Then,
for $\alpha \in (1,2]$ and $|\theta |\leq 2-\alpha $, the characteristic
function of $\mathcal{Y}_{\alpha }^{\theta }(t),t\geq 0$, satisfies the
following equation
\begin{eqnarray}
\partial _{t}\widetilde{u}(\xi ,t) &=&a[-\psi _{\alpha ,\theta }(\xi
)+\lambda t(1-e^{-\psi _{\alpha ,\theta }(\xi )})]\widetilde{u}(\xi ,t)
\label{lem} \\
&&+\frac{\lambda (1-e^{-\psi _{\alpha ,\theta }(\xi )})}{\alpha }|\xi
|^{1-\alpha }e^{-i\,sign(\xi )(\theta -1)\pi /2}\mathcal{F}\{xu(x,t);\xi \},
\notag
\end{eqnarray}%
with initial condition $\widetilde{u}(\xi ,0)=e^{-\psi _{\alpha ,\theta
}(\xi )}.$
\end{lemma}

\begin{proof}
Let
\begin{equation}
u(x,t)=\mathcal{F}^{-1}\left\{ \Phi _{\mathcal{Y}_{\alpha }^{\theta
}(t)}(\xi );x\right\} =\sum_{k=1}^{\infty }e^{-\lambda t}(1-e^{-\lambda
t})^{k-1}p_{\alpha }^{\theta }(x,k+at)  \label{uu}
\end{equation}%
be the transition density of the process $\mathcal{Y}_{\alpha }^{\theta },$
then we apply a conditioning argument and the dominated convergence theorem
(considering (\ref{dom})), to show that%
\begin{eqnarray}
\mathcal{F}\{xu(x,t);\xi \} &=&\mathbb{E}\left\{ \mathbb{E}\left[ \left.
\mathcal{S}_{\alpha }^{\theta }(B(t)+at)e^{i\xi \mathcal{S}_{\alpha
}^{\theta }(B(t)+at)}\right\vert B(t)\right] \right\}  \label{bb} \\
&=&\sum_{k=1}^{\infty }e^{-\lambda t}(1-e^{-\lambda t})^{k-1}\int_{-\infty
}^{+\infty }e^{i\xi x}xp_{\alpha }^{\theta }(x,k+at)dx  \notag \\
&=&\frac{1}{i}\sum_{k=1}^{\infty }e^{-\lambda t}(1-e^{-\lambda
t})^{k-1}\int_{-\infty }^{+\infty }\partial _{\xi }e^{i\xi x}p_{\alpha
}^{\theta }(x,k+at)dx  \notag \\
&=&\frac{1}{i}\sum_{k=1}^{\infty }e^{-\lambda t}(1-e^{-\lambda
t})^{k-1}\partial _{\xi }e^{-\psi _{\alpha ,\theta }(\xi )(k+at)}  \notag \\
&=&i\partial _{\xi }\psi _{\alpha ,\theta }(\xi )e^{-\lambda t-a\psi
_{\alpha ,\theta }(\xi )t}\sum_{k=1}^{\infty }(1-e^{-\lambda
t})^{k-1}(k+at)e^{-\psi _{\alpha ,\theta }(\xi )k}  \notag \\
&=&\alpha i\,sign(\xi )|\xi |^{\alpha -1}e^{i\,sign(\xi )\theta \pi /2}\left[
\frac{\Phi _{\mathcal{Y}_{\alpha }^{\theta }(t)}(\xi )}{1-(1-e^{-\lambda
t})e^{-\psi _{\alpha ,\theta }(\xi )}}+at\Phi _{\mathcal{Y}_{\alpha
}^{\theta }(t)}(\xi )\right] ,  \notag
\end{eqnarray}%
where, in the last step, we have considered that%
\begin{equation*}
\partial _{\xi }\psi _{\alpha ,\theta }(\xi )=\partial _{\xi }|\xi |^{\alpha
}e^{i\,sign(\xi )\theta \pi /2}=\alpha \,sign(\xi )|\xi |^{\alpha
-1}e^{i\,sign(\xi )\theta \pi /2}.
\end{equation*}%
As a check, we notice that, for $\alpha \rightarrow 1^{+},$ $\theta =-1,$ (%
\ref{bb}) reduces to (\ref{dd}), as it must be. By differentiating (\ref{ch}%
) with respect to $t$, we get%
\begin{eqnarray}
\partial _{t}\Phi _{\mathcal{Y}_{\alpha }^{\theta }(t)}(\xi ) &=&-\left[
\lambda +a\psi _{\alpha ,\theta }(\xi )\right] \Phi _{\mathcal{Y}_{\alpha
}^{\theta }(t)}(\xi )+\frac{\lambda e^{-\lambda t-\psi _{\alpha ,\theta
}(\xi )}}{1-(1-e^{-\lambda t})e^{-\psi _{\alpha ,\theta }(\xi )}}\Phi _{%
\mathcal{Y}_{\alpha }^{\theta }(t)}(\xi )  \label{tt} \\
&=&-a\psi _{\alpha ,\theta }(\xi )\Phi _{\mathcal{Y}_{\alpha }^{\theta
}(t)}(\xi )-\lambda (1-e^{-\psi _{\alpha ,\theta }(\xi )})\frac{\Phi _{%
\mathcal{Y}_{\alpha }^{\theta }(t)}(\xi )}{1-(1-e^{-\lambda t})e^{-\psi
_{\alpha ,\theta }(\xi )}},  \notag
\end{eqnarray}%
which, considering (\ref{bb}), coincides with (\ref{lem}). Again, as a
check, (\ref{lem}) reduces to (\ref{pre1}) for $\alpha \rightarrow 1^{+},$ $%
\theta =-1.$
\end{proof}

\

Let $L_{loc}^{1}(\mathbb{R})$ be the space of the locally integrable
function and let $I_{\gamma }^{\nu }$ denote the Feller integral with symbol
(\cite{GOR}, p.341)%
\begin{equation}
\widehat{I_{\gamma }^{\nu }}(\xi )=|\xi |^{-\nu }e^{-i\pi \gamma sign(\xi
)/2},\qquad |\gamma |\leq \left\{
\begin{array}{l}
\nu ,\qquad 0<\nu <1 \\
2-\nu ,\qquad 1<\nu <2%
\end{array}%
\right. ,  \label{int}
\end{equation}%
defined, for $0<\nu <1,$ for functions in $L_{loc}^{1}(\mathbb{R})$ (see
also \cite{HIL}). Since $I_{\gamma }^{\nu }$ is defined only for $\nu \neq 1$%
, for the case $\alpha =2$ we consider instead the Weyl integral with symbol
(\cite{GOR}, p.333)%
\begin{equation}
\widehat{I_{+}^{1}}(\xi )=|\xi |^{-1}e^{i\pi /2sign(\xi )},  \label{rr}
\end{equation}%
which can be written explicitly as%
\begin{equation*}
I_{+}^{1}[f(x)]=\int_{-\infty }^{x}f(z)dz.
\end{equation*}

\begin{theorem}
Let $\alpha \in (1,2)$ and $\theta =2-\alpha .$ Then the density of the
process $\mathcal{Y}_{\alpha }^{\theta }(t)$, $t\geq 0,$ satisfies the
following equation%
\begin{equation}
\partial _{t}u(x,t)=a\left[ \mathcal{D}_{x}^{\alpha ,\theta }+\lambda t(I-%
\mathcal{O}_{1,x}^{\alpha ,\theta })\right] u(x,t)+\frac{\lambda }{\alpha }%
\left( I-\mathcal{O}_{1,x}^{\alpha ,\theta }\,\right) \,I_{1-\alpha
}^{\alpha -1}\left[ xu(x,t)\right] ,  \label{eq3}
\end{equation}%
with initial condition $u(x,0)=p_{\alpha }^{\theta }(x,1),$ while, for $%
\alpha =2,$ $\theta =0,$ we have instead%
\begin{equation}
\partial _{t}u(x,t)=a\left[ \mathcal{\partial }_{x}^{2}+\lambda t(I-e^{%
\mathcal{\partial }_{x}^{2}})\right] u(x,t)+\frac{\lambda }{2}(I-e^{\mathcal{%
\partial }_{x}^{2}})\int_{-\infty }^{x}zu(z,t)dz,  \label{eq4}
\end{equation}%
with $u(x,0)=\varphi (x),$ where $\varphi $ denotes the standard Gaussian
density function.
\end{theorem}

\begin{proof}
We start by noting that, for $\alpha \in (1,2]$, the function $xu(x,t)$
belongs to $L^{1}(\mathbb{R}),$ indeed%
\begin{equation}
\int_{-\infty }^{+\infty }xu(x,t)dx=\sum_{k=1}^{\infty }e^{-\lambda
t}(1-e^{-\lambda t})^{k-1}\int_{-\infty }^{+\infty }xp_{\alpha }^{\theta
}(x,k+t)dx=0,  \label{ll}
\end{equation}%
since the location parameter of $\mathcal{S}_{\alpha }^{\theta }$ is zero by
assumption. We recognize, in the last term of (\ref{lem}), the symbol given
in (\ref{int}), with $\nu =\alpha -1\in (0,1)$ and $\gamma =\theta -1.$ Thus
we introduce the constraint $|\theta -1|\leq \alpha -1$, which must be
considered together with the condition given in (\ref{uno}), i.e. $|\theta
|\leq 2-\alpha $, so that they are jointly satisfied only by $\theta
=2-\alpha .$ Then, considering (\ref{cf}) and (\ref{int}), we can write the
inverse Fourier transform of (\ref{lem}) as in (\ref{eq3}). Equation (\ref%
{eq4}) can be derived analogously from (\ref{lem}), by considering (\ref{rr}%
). The convergence of the integral in (\ref{eq4}) follows from (\ref{ll}).
\end{proof}

\begin{remark}
It is easy to check that equation (\ref{eq3}) reduces to (\ref{pre1}) in the
limit, for $\alpha \rightarrow 1^{+}.$\
\end{remark}

\begin{remark}
The solution to equation (\ref{eq4}) coincides with the transition density
of the process $W(at+B(t)),$ $t>0$. In this case, i.e. for $\alpha =2,$ $%
\theta =0,$ the first two moments exist and we have%
\begin{eqnarray*}
\mathbb{E}\mathcal{Y}_{2}(t) &=&\mathbb{E}\left\{ \mathbb{E}\left[ \left.
W(at+B(t))\right\vert B(t)\right] \right\} =0 \\
Var\mathcal{Y}_{2}(t) &=&\sum_{k=1}^{\infty }e^{-\lambda t}(1-e^{-\lambda
t})^{k-1}\mathbb{E}\left[ W(at+k)\right] ^{2}=2\left( at+\frac{1}{\lambda }%
\right) .
\end{eqnarray*}%
In the general case, for $\alpha \in (0,2)$ we can only evaluate the
fractional moment of order $\gamma \in (-1,\alpha )$, by applying the result
in \cite{SAMO}, p.18,%
\begin{eqnarray}
\mathbb{E}\mathcal{Y}_{\alpha }^{\gamma }(t) &=&\sum_{k=1}^{\infty
}e^{-\lambda t}(1-e^{-\lambda t})^{k-1}\mathbb{E}\mathcal{S}_{\alpha
}^{\gamma }(s)  \label{mom} \\
&=&\frac{\Gamma \left( 1-\frac{\gamma }{\alpha }\right) \cos (\gamma \theta
\pi /2\alpha )(1+\tan ^{2}(\theta \pi /2)^{\gamma /2\alpha }}{\Gamma \left(
1-\gamma \right) \cos (\gamma \pi /2)}\sum_{k=1}^{\infty }e^{-\lambda
t}(1-e^{-\lambda t})^{k-1}\mathbb{(}at+k\mathbb{)}^{\gamma /\alpha }  \notag
\\
&=&\frac{\Gamma \left( 1-\frac{\gamma }{\alpha }\right) \cos (\gamma \theta
\pi /2\alpha )(1+\tan ^{2}(\theta \pi /2)^{\gamma /2\alpha }}{\Gamma \left(
1-\gamma \right) \cos (\gamma \pi /2)}\sum_{j=0}^{\infty }\binom{\gamma
/\alpha }{j}(at)^{\gamma /\alpha -j}\mathbb{E}Z^{j}  \notag
\end{eqnarray}%
where ${Z}_{j}$ is geometric r.v. $Z$ with parameter $e^{-\lambda t}.$
\end{remark}

As far as the tails behavior of the density of $\mathcal{Y}_{\alpha
}^{\theta }(t),$ for any fixed $t,$ we prove that it has regularly varying
tails, with index $\alpha $. Thus it exhibits the same tails behavior of the
stable process and of $\mathcal{X}_{\alpha }^{\theta }(t)$ (see Theorem 14),
even though the constant, in this case, is not linear in time.

\begin{theorem}
For $\alpha \in (0,2)$, we have that%
\begin{equation}
\left\{
\begin{array}{c}
\lim_{x\rightarrow \infty }x^{\alpha }P(\mathcal{Y}_{\alpha }^{\theta
}(t)>x)=\frac{C_{\alpha ,\theta }\frac{1+ate^{-\lambda t}}{e^{-\lambda t}}}{%
\Gamma (1-\alpha )} \\
\lim_{x\rightarrow \infty }x^{\alpha }P(\mathcal{Y}_{\alpha }^{\theta
}(t)<-x)=\frac{C_{\alpha ,\theta }^{\prime }\frac{1+ate^{-\lambda t}}{%
e^{-\lambda t}}}{\Gamma (1-\alpha )}%
\end{array}%
\right. ,\qquad t\geq 0,  \label{cr}
\end{equation}%
where $C_{\alpha ,\theta }=\frac{1}{2}\left[ 1-\frac{\tan (\pi \alpha /2)}{%
\tan (\pi \theta /2)}\right] $ and $C_{\alpha ,\theta }^{\prime }=\frac{1}{2}%
\left[ 1+\frac{\tan (\pi \alpha /2)}{\tan (\pi \theta /2)}\right] $.
\end{theorem}

\begin{proof}
For $\alpha \in (0,1),$ $\theta =-\alpha $, we can write that%
\begin{eqnarray*}
\int_{0}^{+\infty }e^{-\eta x}P(\mathcal{S}_{\alpha ,-\alpha }(at+B(t))>x)dx
&=&\frac{1-\mathbb{E}e^{-\eta \mathcal{S}_{\alpha ,-\alpha }(at+B(t))}}{\eta
} \\
&=&\frac{1-(1-e^{-\lambda t})e^{-\eta ^{\alpha }}-e^{-\lambda t-(at+1)\eta
^{\alpha }}}{\eta \left[ 1-(1-e^{-\lambda t})e^{-\eta ^{\alpha }}\right] } \\
&\sim &\frac{1-(1-e^{-\lambda t})(1-\eta ^{\alpha })-e^{-\lambda t}\left[
1-(at+1)\eta ^{\alpha }\right] }{\eta \left[ 1-(1-e^{-\lambda t})(1-\eta
^{\alpha })\right] } \\
&\sim &\frac{1+ate^{-\lambda t}}{e^{-\lambda t}}\eta ^{\alpha -1},
\end{eqnarray*}%
for any fixed $t$ and for $\eta \rightarrow 0,$ so that we get the first
equation in (\ref{cr}), for $C_{\alpha ,\theta }=1.$ The rest of the proof
follows the same lines of Theorem 14.
\end{proof}


\begin{thebibliography}{99}
\bibitem{ALI} \noindent Alipour M., Beghin L., Rostamy D. (2015),
Generalized fractional nonlinear birth processes, \emph{Method. Comput.
Appl. Probab., }17 (3), 525-540.

\bibitem{ANG} \noindent Angulo J.M., Ruiz-Medina M. D., Anh V. V., Grecksch
W. (2000). Fractional diffusion and fractional heat equation,\emph{\ Adv. in
Appl. Probab}., 32, 1077--1099.

\bibitem{APPL} \noindent Applebaum D. (2009),\textbf{\ }\emph{L\'{e}vy
Processes and Stochastic Calculus}, Cambridge Studies in Advanced
Mathematics, Cambridge.

\bibitem{BEG} \noindent Beghin L. (2014), Geometric stable processes and
fractional differential equation related to them, \emph{Electron. Commun.
Probab.} 19, n. 13, 1--14.

\bibitem{BEGG} \noindent Beghin L. (2015), Fractional gamma and
gamma-subordinated processes, \emph{Stoch. Anal. Appl.,} to appear.

\bibitem{BEG2} \noindent Beghin L., D'Ovidio M. (2014), Fractional Poisson
process with random drift, \emph{Electr. Journ. Probab.,} 19, n.122, 1-26.

\bibitem{CHE} \noindent Chen Z.Q., Meerschaert M., Nane E. (2012),
Space--time fractional diffusion on bounded domains, \emph{J. Math. Anal.
Appl.}, 393, 479--488.

\bibitem{CON} \noindent Cont, R., Tankov, P. (2004). \emph{Financial
Modelling with Jump Processes}, Chapman and Hall/CRC Press, London.

\bibitem{DOV} \noindent D'Ovidio M. (2014), Multidimensional fractional
advection-dispersion equations and related stochastic processes, \emph{%
Electr. Journ. Probab.,} 19, n.61, 1-34.

\bibitem{DATT} \noindent Dattoli G., Ricci P.E., Sacchetti D. (2003),
Generalized shift operators and pseudo-polynomials of fractional order,
\emph{Appl. Math. Comp.,} 141, 215-224.

\bibitem{FEL} \noindent Feller W. (1971),\emph{\ An Introduction Probability
Theory and its Applications, vol.2, II ed., Wiley, New York.}

\bibitem{LUC} \noindent Gorenflo R., Iskenderov A. Luchko Y. (2000), Mapping
between solutions of fractional diffusion-wave equations. \emph{Fract. Calc.
Appl. Anal.} 3, 75-86.

\bibitem{GOR} \noindent Gorenflo R., Kilbas A.A., Mainardi F., Rogosin S.V.
(2014).\emph{\ Mittag-Leffler Functions, Related Topics and Applications}.
Springer.

\bibitem{HIL} \noindent Hilfer R. (2008), Threefold introduction to
fractional derivatives. Anomalous Transport: Foundations and Applications,
R. Klages et al. eds., Wiley, 17 pp.

\bibitem{KIL} \noindent Kilbas A.A., Srivastava H.M., Trujillo J.J. (2006),%
\textbf{\ }\emph{Theory and Applications of Fractional Differential Equations%
}, vol. 204 of North-Holland Mathematics Studies, Elsevier Science B.V.,
Amsterdam.

\bibitem{KOC} \noindent Kochubei A. N. (1990), Diffusion of fractional
order. \emph{Differentsial'nye Uravneniya, }26, 660--670, 733--734.

\bibitem{KOZ} \noindent Kozubowski T.J., Panorska A.K. (1996), On moments
and tail behavior of $\nu $-stable random variables, \emph{Stat. Probab.
Lett.,} 29, 307-315.

\bibitem{KOZ2} \noindent Kozubowski T. J., Rachev S. T.\ (1999), Univariate
geometric stable laws. \emph{J. Comput. Anal. Appl.,} 1 (2), 177--217.

\bibitem{KOZ3} \noindent Kozubowski T. J., Podg\'{o}rski, K., Samorodnitsky
G. (1999), Tails of L\'{e}vy measure of geometric stable random variables,
\emph{Extremes,} 1 (3), 367--378.

\bibitem{MAI} \noindent Mainardi F., Luchko Y.F., Pagnini G. (2001), The
fundamental solution of the space-time fractional diffusion equation, \emph{%
Fract. Calc. Appl. Anal,} 4 (2), 153-192.

\bibitem{MAI2} \noindent Mainardi F., Pagnini G., Saxena R.K. (2005), Fox H
functions in fractional diffusion, \emph{J. Comp. Math. Appl,} 178 (1-2),
321-331.

\bibitem{MAR} \noindent Marichev O.I. (1983), \emph{Handbook of Integral
Transforms of Higher Transcendental Functions, Theory and Algorithmic Tables}%
. Chichester, Ellis Horwood.

\bibitem{MEE} \noindent Meerschaert M., Nane E., Vellaisamy P. (2009),
Fractional Cauchy problems on bounded domains, \emph{Ann. Probab.,} 37 (3),
979--1007.

\bibitem{MEZ} \noindent Metzler, R., Klafter, J. (2004), The restaurant at
the end of the random walk: recent developments in the description of
anomalous transport by fractional dynamics, J. Phys. A, 37, 161--208.

\bibitem{MIJ} \noindent Mijena J.B., Nane E. (2015), Intermittence and
space-time fractional stochastic partial differential equations, \emph{%
Potential Analysis, }forthcoming.

\bibitem{ORB} \noindent Orsingher E., Beghin L. (2009), Fractional diffusion
equations and processes with randomly varying time, \emph{Ann. Probab.,} 37,
(1), 206--249.

\bibitem{ORS} \noindent Orsingher E., Polito F. (2010), Fractional pure
birth processes, \emph{Bernoulli,} 16 (3), 858-881.

\bibitem{ORS2} \noindent Orsingher E., Polito F. (2011), On a fractional
linear birth--death process, \emph{Bernoulli,} 17 (1), 114-137..

\bibitem{SAMO} \noindent Samorodnitsky G., Taqqu M.S (2004), \emph{Stable
Non-Gaussian Random Processes}, Chapman and Hall, New York.

\bibitem{SATO} \noindent Sato K.I. (1999), \emph{L\'{e}vy Processes and
Infinitely Divisible Distributions,} Cambridge Studies in Adv. Math.,
Cambridge Univ. Press.

\bibitem{SCA} \noindent Scalas, E. (2004). Five years of continuous-time
random walks in econophysics, \emph{Proceedings of WEHIA} (Kyoto, 2004) (A.
Namatame, ed.). Springer.

\bibitem{SHA} Shanbhag D. N. , Sreehari M., (1977), On certain
self-decomposable distributions, \emph{Zeit. Wahrsch. Verw. Gebiete}, 38,
217--222.

\bibitem{SCH} \noindent Schneider W. R., Wyss W. (1989). Fractional
diffusion and wave equations, \emph{J. Math. Phys., }30, 134--144.

\bibitem{VERW} \noindent Veraart A.E.D. , Winkel M., (2010), Time change,
\emph{Encyclopedia of Quantitative Finance,} Wiley, 1812-1816.

\bibitem{WYS} \noindent Wyss W. (1986), The fractional diffusion equation,
\emph{J. Math. Phys}., 27, 2782--2785.
\end{thebibliography}
\end{document}